\documentclass[12pt]{article}
\usepackage{amsmath}
\usepackage{amsfonts}
\usepackage{amssymb}
\usepackage{graphicx}%
\providecommand{\U}[1]{\protect \rule{.1in}{.1in}}
\newtheorem{theorem}{Theorem}[section]

\newtheorem{definition}[theorem]{Definition}

\newtheorem{lemma}{Lemma}[section]

\newtheorem{remark}[theorem]{Remark}

\newenvironment{proof}[1][Proof]{\noindent \textbf{#1.} }{\  $\Box$}

\title{Backward  stochastic Volterra integral equations associated with a L\'{e}vy process and applications
\footnote{The work  is supported partially by the National Natural
Science Foundation of China (61273128, 11371029) and a Project of
Shandong  Province Higher Educational Science and Technology Program
(J13LI06).}}

\author{Wen LU\footnote{\emph{Email
address:} llcxw@163.com}\\
        \footnotesize{School of Mathematics,  Yantai University, Yantai 264005, China
 }}

\begin{document}
\date{}
\maketitle
\begin{abstract}
 In this paper,  we study a class of backward stochastic Volterra integral equations driven
 by Teugels martingales associated with an
independent L\'{e}vy process and an independent Brownian motion
(BSVIELs). We prove the existence and uniqueness as well as
stability
 of the adapted M-solutions for those equations. Moreover,  a
  duality principle and then a comparison theorem are established. As an application,  we derive a class
  of
  dynamic risk measures by means of M-solutions of certain BSVIELs.
\end{abstract}



\section{Introduction}
 The general nonlinear case backward stochastic differential
 equations (BSDEs), i.e., equations in form
\begin{eqnarray}\label{bsde:1}
Y(t)=\xi+\int_t^T f(s, Y(s), Z(s)ds-\int_t^T Z(s)dW_{s},\quad
t\in[0, T],
\end{eqnarray}
was first introduced by Pardoux and Peng \cite{PP90} in 1990, they
proved the existence and uniqueness of solutions for BSDEs under
Lipschitz conditions. Since then, a lot of work have been devoted to
the study of the theory of BSDEs as well as to their applications.
This is due to the connections of BSDEs with mathematical finance as
well as stochastic optimal control and stochastic games (see e.g.,
\cite{KPQ97}, \cite{Peng93}, \cite{HL95}).

 In Nualart and Schoutens  \cite{NS00}, the authors gave a
martingale representation theorem associated to Teugels martingales
corresponding to a  L\'{e}vy process. Furthermore, Nualart and
Schoutens \cite{NS01} studied the corresponding BSDEs associated to
a L\'{e}vy process. The results were important from a pure
mathematical point of view as well as in the world of finance. It
could be used for the purpose of option pricing in a  L\'{e}vy
market and the partial differential equation which provided an
analogue of the famous Black-Scholes partial differential equation.
Following that, Bahlali et al. \cite{BEE} considered the BSDEs
driven by a Brownian motion and the martingales of Teugels
associated with an independent L\'{e}vy process, having a Lipschitz
or a locally Lipschitz coefficient.

On the other hand, stochastic Volterra equation had been
investigated by Berger and Mizel in \cite{BM801} and \cite{BM802},
Protter in \cite{Pro85} and Pardoux and Protter in \cite{ParPro90}.
As a natural generalization of the BSDE theory, Lin \cite{Lin02}
firstly considered the solvability of the adapted solution for
backward stochastic Volterra integral equations (BSVIEs) with
uniform Lipschtz coefficient of the form
\begin{eqnarray}\label{bsde:2}
Y(t)&=&\xi+\int_t^T f(s, Y(s), Z(t,s))ds \nonumber\\&& -\int_t^T
[g(t,s, Y(s))+Z(t,s)]dW_{s},\quad t\in[0, T].
\end{eqnarray}
Following it, Aman and N'Zi \cite{AmanN'Zi05} considered the same
equation and weakened the uniform Lipschtz condition on the
coefficient to a local one. Thereafter, Yong \cite{Yong06} extended
the equations (\ref{bsde:2}) to a generalized form. For the more
general cases of BSVIEs (\ref{bsde:2}), Anh and Yong
\cite{AnhYong06} and Yong (\cite{Yong07},  \cite{Yong08}) studied
them and gave its applications in stochastic optimal control,
mathematics finance and risk management, where the notion of
M-solution was introduced to ensure the unique solvability of the
adapted solution. Recently, Ren \cite{Ren09} established the
well-posedness of adapted M-solutions for BSVIEs driven by both
Brownian motion and a Poisson random measure.

  Motivated by above
works, it is natural and necessary to consider the backward
stochastic Volterra integral equations driven by a standard Brownian
motion and the Teugels martingales associated with an independent
L\'{e}vy process (BSVIELs). We first show the existence and
uniqueness of M-solutions for those equations. Then, a duality
principle between the linear BSVIELs and the linear forward
stochastic Volterra integral equations driven by the same Brownian
motion and the Teugels martingales associated with an independent
L\'{e}vy process (FSVIELs) is presented. Further, as an important
application of the duality principle, we establish a comparison
theorem for M-solutions of BSVIELs. Finally, a class of
  dynamic risk measures are derived  by means of M-solutions of one kind of BSVIELs.
We would like to point that we adopt the similar method in Anh and
Yong \cite{AnhYong06}, but the dynamic system is different from
\cite{AnhYong06}.

 The rest of the paper is organized as follows. In Section 2, we introduce some preliminaries, then we prove the existence and
uniqueness of the adapted M-solutions for  BSVIELs in Section 3. In
Section 4, we establish a duality principle between linear BSVIELs
and linear FSVILs as well as a comparison theorem for M-solutions of
BSVIELs. In Section 5, a class of dynamic coherent risk measures be
derived by means of M-solutions of certain BSVIELs.

 \section{Preliminaries}

Given $T>0$ a fixed real number. Let's first introduce the following
two mutually independent processes:
\begin{itemize}
\item $\{W_t: t\in[0, T]\}$: a standard
Brownian motion in $\mathbb{R}^d$;
 \item A $\mathbb{R}$-valued L\'{e}vy process
$(L_t)_{0\leq t\leq T}$ corresponding to a standard L\'{e}vy measure
$\nu$ satisfying the following conditions:

(i) $\int_{\it R}(1\wedge y^2)\nu(dy)<\infty$,

(ii) $\int_{]-\varepsilon, \varepsilon[^c}\mbox{\large
e}^{\lambda|y|}\nu(dy)<\infty$, for every $\varepsilon>0$ and for
some $\lambda>0$.
 \end{itemize}

Let $(\Omega, \mathcal{F}, \mathbb{F}, P)$ be a complete filtered
probability space, the filtration
$\mathbb{F}=\{\mathcal{F}_{t}\}_{0\leq t\leq T}$ is generated by the
two processes given above, i.e.,
$$\mathcal{F}_t=\sigma\{W_{s}, 0\leq s\leq t\} \vee \sigma\{L_s, 0\leq s\leq t\}\vee\mathcal {N},$$
where $\mathcal {N}$ is the set of all $P$-null subsets of
$\mathcal{F}$.

 We define:
 \begin{itemize}
\item $\mathbb{L}^2_{\mathcal {F}_T}(0,T; \mathbb{R}^n)=\{ \psi:
[0,T]\times \Omega \rightarrow \mathbb{R}^n \;|\; \psi(\cdot)\;
 \mbox{is}\;   \mathcal {B}([0,T])\otimes  \mathcal{F}_T$-measurable such that ${\bf
 E}\int_0^T|\psi(t)|^2dt<\infty\}$;
 \item $\mathbb{L}^2_{\mathbb{F}}(0,T;
\mathbb{R}^n)=\{\psi(\cdot)\in\mathbb{L}^2_{\mathcal
{F}_T}(0,T;\mathbb{R}^n) | \psi(\cdot)\;\mbox{is}\;  \mathbb{F}
\mbox{-adapted}\}$;
 \item ${\ell}^2=\{x=(x^{(i)})_{i\geq 1}\; |\;
\|x\|=\left[\sum_{i=1}^{\infty}(x^{(i)})^2\right]^{\frac{1}{2}} <
\infty\}$.
 \end{itemize}
\begin{remark}
In all the definitions  of the relevant spaces in this paper,
$[0,T]$ can be replaced by any $[R,S]$ with $0\leq R< S\leq T$.
\end{remark}

In what follows, for any $0\leq R< S\leq T$, we denote
$$\Delta[R,S]=\{(t,s)\in[R,S]^2 \;|\; R\leq s\leq t\leq S\},$$
$$\Delta^c[R,S]=\{(t,s)\in[R,S]^2 \;|\; R\leq t<s \leq S\}.$$
For simplicity, we denote $\Delta[0,T]=\Delta,\;
\Delta^c[0,T]=\Delta^c$.


We denote by $(H^{(i)})_{i\geq1}$ the Teugels martingales associated
with the L\'{e}vy process $\{L_t: \; t\in[0, T]\}$. More precisely
 \begin{eqnarray*}
 H_t^{(i)}=c_{i,i}Y_t^{(i)}+c_{i,i-1}Y_t^{(i-1)}+\cdots+c_{i,1}Y_t^{(1)},
 \end{eqnarray*}
where
$Y_t^{(i)}=L_t^{(i)}-\mathbb{E}[L_t^{(i)}]=L_t^{(i)}-t\mathbb{E}[L_1^{(i)}]$
for all $i\geq 1$ and $L_t^{(i)}$ are so called power-jump
processes, i.e., $L_t^{(1)}=L_t$ and $L_t^{(i)}=\sum_{0\leq s\leq
t}(\Delta L_t)^i$ for $i\geq 2$. Here, for any process $x(t)$, we
denote by $x(t-)=\lim_{s\rightarrow t-}x(s)$ and $\Delta
x_t=x(t)-x(t-)$.

 It was shown in  \cite{NS00} that the coefficients $c_{i,k}$ correspond
to the orthonormalization of the polynomials $1, x, x^2, \dots$ with
respect to the measure $\mu(dx)=x^2\nu(dx)+\sigma^2\delta_0(dx)$:
\begin{eqnarray*}
 q_{i-1}=c_{i,i}x^{i-1}+c_{i,i-1}x^{i-2}+\cdots+c_{i,1}.
 \end{eqnarray*}
We set
\begin{eqnarray*}
p_{i}(x)=xq_{i-1}(x)=c_{i,i}x^{i}+c_{i,i-1}x^{i-1}+\cdots+c_{i,1}x.
 \end{eqnarray*}
The martingales $(H^{(i)})_{i\geq1}$ can be chosen to be pairwise
strongly orthonormal martingales. Furthermore, $[H^{(i)}, H^{(j)}],
i\neq j$, and $\{[H^{(i)}, H^{(j)}]_t-t\}_{t\geq 0}$ are uniformly
integrable martingales with initial value 0, i.e., $\langle H^{(i)},
H^{(j)}\rangle_t=\delta_{ij}t$.

 Throughout this paper, we consider the following BSVIEL:
\begin{eqnarray}\label{bsviel:1}
Y(t)&=&\psi(t)+\int_t^Tf(t,s,Y(s-),Z(t,s),Z(s,t),U(t,s),U(s,t))ds\nonumber\\&&+\int_t^TZ(t,s)dW_{s}
-\sum_{i=1}^{\infty}\int_t^TU^{(i)}(t,s)dH_s^{(i)}, \, 0\leq t\leq
T,
\end{eqnarray}
where $\psi(\cdot)\in \mathbb{L}^2_{\mathcal {F}_T}(0, T;
\mathbb{R})$ and $f: \Delta^c\times \mathbb{R}\times
\mathbb{R}^d\times
 \mathbb{R}^d\times \ell^2\times\ell^2\times\Omega\rightarrow \mathbb{R}$ is a given map.

We denote by $\mathbb{L}^2(0,T;\mathbb{L}^2_{\mathbb{F}}(0,T;
\mathbb{R}^d))$ the space  of all processes $Z:
[0,T]^2\times\Omega\rightarrow \mathbb{R}^d$  such that for almost
all $t\in[0,T]$, $Z(t,\cdot)\in\mathbb{L}^2_{\mathbb{F}}(0,T;
\mathbb{R}^d)$ satisfying ${\bf
E}\int_0^T\int_0^T|Z(t,s)|^2dsdt<\infty$. We denote by
${\ell}^2(0,T;\mathbb{L}^2_{\mathbb{F}}(0,T; \mathbb{R}))$ the space
of processes $U: [0,T]^2\times\Omega \rightarrow {\ell}^2$ such that
for each $i\geq1$ and almost all $t\in[0,T]$,
$U^{(i)}(t,\cdot)\in\mathbb{L}^2_{\mathbb{F}}(0,T; \mathbb{R})$
satisfying
$\|U\|^2=\sum_{i=1}^{\infty}\mathbb{E}\int_0^T\int_0^T|U^{(i)}(t,s)|^2dsdt<\infty$.

We make the following assumptions:
\\
(\textbf{H1})  Let $f: \Delta^c\times \mathbb{R}\times
\mathbb{R}^d\times
 \mathbb{R}^d\times \ell^2\times\ell^2\times\Omega\rightarrow \mathbb{R}$ be $\mathcal
{B}(\Delta^c\times \mathbb{R}\times  \mathbb{R}^d\times
\mathbb{R}^d\times \ell^2\times\ell^2)\otimes\mathcal
{F}_T$-measurable such that $s\rightarrow f(t,s,y,z,\eta,u,\zeta)$
is $\mathbb{F}-$adapted for all $(t,y,z,\eta,u,\zeta)\in[0,T]\times
\mathbb{R}\times  \mathbb{R}^d\times
 \mathbb{R}^d\times \ell^2\times\ell^2$ and
\begin{eqnarray*}
E\int_0^T\left(\int_t^T|f_0(t,s)|ds\right)^2dt<\infty,
\end{eqnarray*}
where $f_0(t,s)\equiv f(t,s,0,0,0,0,0)$. Moreover, for any
$(y_1,z_1,\eta_1, u_1, \zeta_1)$ and $(y_2,z_2,\eta_2, u_2,
\zeta_2)\in \mathbb{R}\times \mathbb{R}^d\times
 \mathbb{R}^d\times \ell^2\times \ell^2$, it holds
\begin{eqnarray}
&&|f(t,s,y_1,z_1,\eta_1, u_1, \zeta_1)-f(t,s,y_2,z_2,\eta_2, u_2,
\zeta_2)| \nonumber\\&\leq& L_y(t,s)|y_1-y_2|
 +L_z(t,s)|z_1-z_2|+L_{\eta}(t,s)|\eta_1-\eta_2|
 \nonumber\\&&
 +L_{u}(t,s)\|u_1-u_2\|+L_{\zeta}(t,s)\|\zeta_1-\zeta_2\|,
\end{eqnarray}
 where the coefficients $L_y(t,s),L_z(t,s),L_{\eta}(t,s),L_{u}(t,s)$
and $L_{\zeta}(t,s)$ are determined functions  from $\Delta^c$ to
$\mathbb{R}$ such that
\begin{eqnarray}\label{condt:1}
&&\sup_{t\in[0,T]}\int_t^T[L_y^2(t,s)+L_{\eta}^2(t,s)+L_{\zeta}^2(t,s)]ds<
\infty,\\&&
\label{condt:2}\sup_{t\in[0,T]}\int_t^T[L_z^2(t,s)+L_{u}^2(t,s)]ds<1.
\end{eqnarray}

Let's  give the notion of  M-solution of BSVIEL (\ref{bsviel:1}).

\begin{definition} A triple of processes $(Y(\cdot), Z(\cdot,\cdot), U(\cdot,
\cdot))\in \mathbb{L}^2_\mathbb{F}(0,T; \mathbb{R})\times
\mathbb{L}^2(0,T;\mathbb{L}^2_\mathbb{F}(0,T; \mathbb{R}^d))\times
\ell^2(0,T;\mathbb{L}^2_\mathbb{F}(0,T; \mathbb{R}))$ is called an
adapted M-solution of BSVIEL (\ref{bsviel:1}) if (\ref{bsviel:1}) is
satisfied in the It\^{o}'s sense for almost all $0\leq t\leq T$, and
it holds that
\begin{eqnarray}
Y(t)=EY(t)+\int_0^t Z(t,s) dW_{s}
+\sum_{i=1}^\infty\int_0^tU^{(i)}(t,s)dH^{(i)}_{s}.
\end{eqnarray}
\end{definition}

\section{Existence and uniqueness of solution}
\begin{theorem}\label{theorem:1}
 Suppose (\textbf{H1}) holds. Then for any $\psi(\cdot)\in
\mathbb{L}^2_{\mathcal {F}_T}(0, T; \mathbb{R})$, the BSVIEL
(\ref{bsviel:1}) has a unique M-solution $(Y(\cdot), Z(\cdot,\cdot),
U(\cdot,\cdot))\in \mathbb{L}^2_\mathbb{F}(0,T; \mathbb{R})\times
\mathbb{L}^2(0,T;\mathbb{L}^2_\mathbb{F}(0,T; \mathbb{R}^d))\times
\ell^2(0,T;L^2_\mathbb{F}(0,T; \mathbb{R}))$.
\end{theorem}
\begin{proof} From (\ref{condt:1}) and (\ref{condt:2}), we know that there exists a sequence
$0=T_0<T_1<\cdots<T_{k-1}<T_k=T$ and $\delta\in(0,1)$ satisfying:
\begin{eqnarray}\label{condt:3}
\sup_{t\in[T_{i-1},T_i]}\int_t^{T_i}[L_y^2(t,s)+L_{\eta}^2(t,s)+L_{\zeta}^2(t,s)]ds\leq
\frac{1-\delta}{4}, \, 1\leq i\leq k,
\end{eqnarray}
\begin{eqnarray}\label{condt:4}
\sup_{t\in[T_{i-1},T_i]}\int_t^{T_i}[L_z^2(t,s)+L_{u}^2(t,s)]ds\leq
\frac{1-\delta}{4}, \, 1\leq i\leq k.
\end{eqnarray}

We split the rest of the proof into several steps.

\emph{Step 1}. The existence and uniqueness of M-solution for
BSVIELs (\ref{bsviel:1}) on $[T_{k-1}, T]$.

Let $\mathcal {M}^2[T_{k-1}, T]$ be the subspace of $(y(\cdot),
z(\cdot,\cdot), u(\cdot,\cdot))\in \mathbb{L}^2_\mathbb{F}(T_{k-1},
T; \mathbb{R})\times \mathbb{L}^2(T_{k-1}, T;
\mathbb{L}^2_\mathbb{F}(T_{k-1}, T; \mathbb{R}^d))\times
\ell^2(T_{k-1}, T; \mathbb{L}^2_\mathbb{F}(0,T; \mathbb{R}))$ such
that
\begin{eqnarray}
y(t)=E y(t)+\int_0^{t}z(t,s)dW_{s}+\sum_{i=1}^\infty\int_0^{t}
u^{(i)}(t,s)dH^{(i)}_{s},\quad t\in[T_{k-1}, T].
\end{eqnarray}
Furthermore, for any $(y(\cdot),z(\cdot,\cdot),u(\cdot,\cdot))\in
\mathcal {M}^2[T_{k-1}, T]$, $(t,r)\in\Delta$ and $t\in [T_{k-1},
T]$, we have
\begin{eqnarray}\label{inequality:1}
 &&E\int_r^{t}|z(t,s)|^2ds+E\int_r^{t}\|u(t,s)\|^2ds\nonumber \\&\leq&
E\int_0^{t}|z(t,s)|^2ds+E\int_0^{t}\|u(t,s)\|^2ds\nonumber
\\&\leq&
E|y(t)-E y(t)|^2 \leq E|y(t)|^2.
\end{eqnarray}

For every $(y(\cdot),z(\cdot,\cdot),u(\cdot,\cdot))\in \mathcal
{M}^2[T_{k-1}, T]$  and $t\in [T_{k-1}, T]$, we denote
\begin{eqnarray}
\overline{\varphi}(t) =\varphi(t)+\int_t^{T}f(t,s, Y(s-),
z(t,s),z(s,t), u(t,s), u(s,t))ds.
\end{eqnarray}
By (\ref{condt:3}) and (\ref{condt:4}), from (\textbf{H1}) and
Cauchy-Schwartz inequality, for any $t\in [T_{k-1}, T]$, we have
\begin{eqnarray}
|\overline{\varphi}(t)|^2&\leq&
C\left[|\varphi(t)|^2+\left(\int_t^{T}f_0(t,s)ds\right)^2+\int_t^{T}|y(s)|^2ds\right]\nonumber\\&&+
(1-\delta^2)\left[\int_t^{T}|z(t,s)|^2ds+\int_t^{T}|z(s,t)|^2ds\right.\nonumber\\&&\left.+\int_t^{T}\|u(t,s)\|^2ds+\int_t^{T}\|u(s,t)\|^2ds\right].
\end{eqnarray}
 Hereafter $C$   is a generic positive constant which may be
different from line to line.

 Noting (\ref{inequality:1}), for any $r\in[T_{k-1}, T]$, we have
\begin{eqnarray*}
&&E\int_{r}^{T}|\overline{\varphi}(t)|^2dt\\&\leq& C
E\left[\int_r^{T}|\varphi(t)|^2dt+\int_r^{T}\left(\int_t^{T}|f_0(t,s)|ds\right)^2dt
+\int_r^{T}\int_t^{T}|y(s)|^2ds\right]dt\\&&+
(1-\delta^2)E\left[\int_r^{T}|y(t)|^2dt+\int_r^{T}\int_t^{T}|z(t,s)|^2dsdt\right.
\nonumber\\&&\left.+\int_r^{T}\int_t^{T}\|u(t,s)\|^2dsdt\right].
\end{eqnarray*}
which implies that $\overline{\varphi}(\cdot)\in
\mathbb{L}^2_{\mathcal {F}_T}(T_{k-1}, T; \mathbb{R})$.   Then, for
any $t\in[T_{k-1}, T]$, by the
 martingale representation theorem in   Bahlari et al. \cite{BEE},
 there exists a unique pair of processes $(Z(\cdot, \cdot), U(\cdot, \cdot))\in
 \mathbb{L}^2(0,T; \mathbb{L}^2_\mathbb{F}(0,T; \mathbb{R}^d))\times
\ell^2(0,T;\mathbb{L}^2_\mathbb{F}(0,T; \mathbb{R}))$ such that
\begin{eqnarray}
\overline{\varphi}(t)&=&E\overline{\varphi}(t)+\int_0^TZ(t,s)dW_{s}+\sum_{i=1}^\infty\int_0^{T}
U^{(i)}(t,s)dH^{(i)}_{s}.
\end{eqnarray}
Let
\begin{eqnarray}
Y(t)&=&E\overline{\varphi}(t)+\int_0^tZ(t,s)dW_{s}+\sum_{i=1}^\infty\int_0^{T}
U^{(i)}(t,s)dH^{(i)}_{s},
\end{eqnarray}
we then get
\begin{eqnarray*}
Y(t)&=&\overline{\varphi}(t)-\int_t^TZ(t,s)dW_{s}-\sum_{i=1}^\infty\int_t^{T}
U^{(i)}(t,s)dH^{(i)}_{s}
\\&=&
\varphi(t)+\int_t^{T}f(t,s, Y(s-), z(t,s),z(s,t), u(t,s),
u(s,t))ds\\&&-\int_t^TZ(t,s)dW_{s}-\sum_{i=1}^\infty\int_t^{T}
U^{(i)}(t,s)dH^{(i)}_{s}.
\end{eqnarray*}
Thus, we get a unique adapted M-solution
$(Y(\cdot),Z(\cdot,\cdot),U(\cdot,\cdot))$ for BSVIEL
(\ref{bsviel:1}) on $[T_{k-1}, T]$. Clearly,
$(Y(\cdot),Z(\cdot,\cdot),U(\cdot,\cdot))\in \mathcal {M}^2[T_{k-1},
T]$.

Next, we prove the uniqueness of adapted M-solution.

Let $(Y(\cdot),Z(\cdot,\cdot),U(\cdot,\cdot))$ and
$(\overline{Y}(\cdot),\overline{Z}(\cdot,\cdot),\overline{U}(\cdot,\cdot))$
be  two adapted M-solutions of BSVIEL (\ref{bsviel:1}) on $[T_{k-1},
T]$. For all $t\in [T_{k-1}, T]$, by (\textbf{H1}) and
Cauchy-Schwartz inequality, we're able to obtain
\begin{eqnarray}
&&E|Y(t)-\overline{Y}(t)|^2+E\int_t^{T}|Z(t,s)-\overline{Z}(t,s)|^2ds\nonumber\\&&+E\int_t^{T}\|U(t,s)-\overline{U}(t,s)\|^2ds
\nonumber\\&\leq& CE  \int_t^{T}|Y(s)-\overline{Y}(s)|^2ds
+(1-\delta^2)E\left[\int_t^{T}|Z(t,s)-\overline{Z}(t,s)|^2ds\right.\nonumber\\&&
\left.+E\int_t^{T}|Z(s,t)-\overline{Z}(s,t)|^2ds+E\int_t^{T}\|U(t,s)-\overline{U}(t,s)\|^2ds\right.\nonumber\\&&
\left.+E\int_t^{T}\|U(s,t)-\overline{U}(s,t)\|^2ds\right]
\end{eqnarray}
 Similar to (\ref{inequality:1}), we have
\begin{eqnarray}
&&E\int_r^{t}|Z(t,s)-\overline{Z}(t,s)|^2ds+E\int_r^{t}\|U(t,s)-\overline{U}(t,s)\|^2ds\nonumber\\&&\leq
E|Y(t)- \overline{Y}(t)|^2, \;(t,r)\in\Delta
\end{eqnarray}
Hence, for $r\in [T_{k-1}, T)$, we have
\begin{eqnarray}
&&E\int_r^T|Y(t)-\overline{Y}(t)|^2dt+E\int_r^T\int_t^{T}|Z(t,s)-\overline{Z}(t,s)|^2dsdt
\nonumber\\&&+E\int_r^T\int_t^{T}\|U(t,s)-\overline{U}(t,s)\|^2dsdt
\nonumber\\&\leq& CE \int_r^T
\int_t^{T}|Y(s)-\overline{Y}(s)|^2dsdt\nonumber\\&&
+(1-\delta^2)\left[E\int_r^T|Y(t)-\overline{Y}(t)|^2dt+E\int_r^T\int_t^{T}|Z(t,s)-\overline{Z}(t,s)|^2dsdt\right.\nonumber\\&&
\left.+E\int_r^T\int_t^{T}\|U(t,s)-\overline{U}(t,s)\|^2dsdt
\right].
\end{eqnarray}
Then the uniqueness is an immediate consequence of Gronwall's
inequality.

 \emph{Step 2}. Solvability of a stochastic integral
equation on $[T_{k-1}, T]$.

For $(t,s)\in [T_{k-2}, T_{k-1}]\times[T_{k-1}, T]$, by Step 1, we
know that the values $Y(s), Z(s,t)$ and $U(s,t)$ are already
determined. Hence, for $ (t,s,z,u)\in[T_{k-2},
T_{k-1}]\times[T_{k-1}, T]\times  \mathbb{R}^d\times \ell^2$, we can
define
\begin{eqnarray}
f^{k-1}(t,s,z,u)=f(t,s,Y(s-),z,Z(s,t),u,U(s,t)).
\end{eqnarray}
We now consider the following stochastic integral equation :
\begin{eqnarray}
\varphi^{k-1}(t)&=&\varphi(t)+\int_{T_{k-1}}^Tf^{k-1}(t,s,Z(t,s),U(t,s))ds-\int_{T_{k-1}}^TZ(t,s)dW_{s}
\nonumber \\&& -\sum_{i=1}^\infty\int_{T_{k-1}}^T
U^{(i)}(t,s)dH^{(i)}_{s},\quad t \in[T_{k-2}, T_{k-1}],
\end{eqnarray}
By (\textbf{H1}), the above equation admits a unique solution
$(\varphi^{k-1}(\cdot),Z(\cdot, \cdot),U(\cdot, \cdot))$ such that
$\varphi^{k-1}(t)$ being $\mathcal {F}_{T_{k-1}}$-adapted. This
uniquely determines the values $Z(t,s)$ and $U(t,s)$ for
$(t,s)\in[T_{k-2}, T_{k-1}]\times[T_{k-1}, T]$.

\emph{Step 3}. Complete the proof by induction.

By the previous two steps, we have determine the values  $Y(t)$ for
$t\in [T_{k-1}, T]$, and the values $Z(t,s)$  and $U(t,s)$ for
$(t,s)\in\left([T_{k-1}, T]\times[0, T]\right)\cup\left([T_{k-2},
T_{k-1}]\times[T_{k-1}, T]\right)$. From the definition of
$f^{k-1}(t,s,z,u)$, one can see that $(\varphi^{k-1}(\cdot),Z(\cdot,
\cdot),U(\cdot, \cdot))$ satisfies
\begin{eqnarray*}
\varphi^{k-1}(t)&=&\varphi(t)+\int_{T_{k-1}}^Tf(t,s,Y(s-),Z(t,s),Z(s,t),U(t,s),U(s,t))ds\nonumber
\\&&-\int_{T_{k-1}}^TZ(t,s)dW_{s}
 -\sum_{i=1}^\infty\int_{T_{k-1}}^T
U^{(i)}(t,s)dH^{(i)}_{s}, t \in[T_{k-2}, T_{k-1}].
\end{eqnarray*}
For $ t \in[0, T_{k-1}]$, we consider the following equation:
\begin{eqnarray}\label{bsde:28}
Y(t)&=&\varphi^{k-1}(t)+\int_t^{T_{k-1}}f(t,s,Y(s-),Z(t,s),Z(s,t),U(t,s),U(s,t))ds\nonumber
\\&&-\int_t^{T_{k-1}}Z(t,s)dW_{s}
 -\sum_{i=1}^\infty\int_t^{T_{k-1}}
U^{(i)}(t,s)dH^{(i)}_{s}.
\end{eqnarray}
Note that $\varphi^{k-1}(t)$ is $\mathcal {F}_{T_{k-1}}$-adapted.
From Step 1, we are able to prove that (\ref{bsde:28}) is solvable
on $[T_{k-2}, T_{k-1}]$, then the values $Y(t)$ for $t\in[T_{k-2},
T_{k-1}]$,  the values $Z(t,s)$ and $U(t,s)$ for $(t,s)\in [T_{k-2},
T_{k-1}]\times[0, T_{k-1}]$ are determined. Therefore, we obtain the
values $Y(t)$ for $t\in[T_{k-2}, T]$, and the values $Z(t,s)$  and
$U(t,s)$ for $(t,s)\in [T_{k-2}, T]\times[0, T]$. Moreover,  for
$t\in [T_{k-2}, T_{k-1}]$, we have
\begin{eqnarray}
Y(t)&=&\varphi^{k-1}(t)+\int_t^{T_{k-1}}f(t,s,Y(s-),Z(t,s),Z(s,t),U(t,s),U(s,t))ds\nonumber
\\&&-\int_t^{T_{k-1}}Z(t,s)dW_{s}
-\sum_{i=1}^\infty\int_t^{T_{k-1}}
U^{(i)}(t,s)dH^{(i)}_{s}\nonumber\\&=&
\varphi(t)+\int_t^Tf(t,s,Y(s-),Z(t,s),Z(s,t),U(t,s),U(s,t))ds\nonumber
\\&&-\int_t^TZ(t,s)dW_{s}
 -\sum_{i=1}^\infty\int_t^T U^{(i)}(t,s)dH^{(i)}_{s}.
\end{eqnarray}
Thus, the equation is solvable on $[T_{k-2}, T]$. We then complete
the proof by induction.
\end{proof}

Further, we have the following stable result.
\begin{theorem} Let $\overline{f}: \Delta^c\times \mathbb{R}\times
\mathbb{R}^d\times
 \mathbb{R}^d\times \ell^2\times\ell^2\times\Omega\rightarrow \mathbb{R}$
satisfy (H1) and $\overline{\psi}(\cdot) \in \mathbb{L}^2_{\mathcal
{F}}(0, T; \mathbb{R})$. Let $(\overline{Y}(\cdot),
\overline{Z}(\cdot,\cdot), \overline{U}(\cdot,\cdot))\in
\mathbb{L}^2_\mathbb{F}(0,T; \mathbb{R})\times
\mathbb{L}^2(0,T;\mathbb{L}^2_\mathbb{F}(0,T; \mathbb{R}^d))\times
\ell^2(0,T;\mathbb{L}^2_\mathbb{F}(0,T; \mathbb{R}))$ be the unique
M-solution of BSVIEL (\ref{bsviel:1}) corresponding to
 $(\overline{\psi}(\cdot),\overline{f})$. Then, for all $r\in[0,T]$, we have
\begin{eqnarray}
&&E\int_r^T|Y(t)-\overline{Y}(t)|^2dt+E\int_r^T\int_r^{T}|Z(t,s)-\overline{Z}(t,s)|^2dsdt
\nonumber\\&&+E\int_r^T\int_r^{T}\|U(t,s)-\overline{U}(t,s)\|^2dsdt
\nonumber\\&\leq& C\left[E \int_r^T
|\psi(t)-\overline{\psi}(t)|^2dt\right.\nonumber\\&&
\left.+\int_r^T\left(\int_r^{T}|f(t,s,Z(t,s),Z(s,t),U(t,s),
U(s,t))\right.\right.\nonumber\\&&
\left.\left.\qquad\qquad-\overline{f}(t,s,Z(t,s),Z(s,t),U(t,s),
U(s,t))|ds\right)^2dt \right].
\end{eqnarray}
\end{theorem}
\begin{proof}
The proof is similar to  Step 1 of the proof of Theorem
\ref{theorem:1}.
\end{proof}

\section{Duality principle and comparison theorem }

In this section, we establish a duality principle between linear
BSVIELs and linear FSVIELs. As an application of the duality
principle, a comparison theorem  for M-solutions of certain BSVIELs
is given.

Consider the following BSVIEL:
\begin{eqnarray}
Y(t)&=&\varphi(t)+\int_t^T[B_0(t,s)Y(s-)+
B(t,s)Z(s,t)+\sum_{i=1}^{\infty}C^{(i)}(t,s)U^{(i)}(s,t)]ds\nonumber\\&&
-\int_t^TZ(t,s)dW_{s}-\sum_{i=1}^{\infty}\int_t^TU^{(i)}(t,s)dH_s^{(i)}.
\end{eqnarray}
Here $B_0(\cdot,\cdot)$, $B(\cdot,\cdot)=(B_1(\cdot,\cdot),\cdots,
B_d(\cdot,\cdot))^T$ and  $\left(C^{(i)}(\cdot,\cdot)\right)_{i\geq
1}$ satisfying the following assumption:
\\(\textbf{H2}) For each $j=0,1,\dots, d$, the process $B_j: \Delta^c\times \Omega\rightarrow \mathbb{R}$ such that for each $t\in[0,T]$,
$B_j(t, s)$ is $F$-adapted, and  $\sup_{(t,s)\in
\Delta^c}\mbox{esssup}_{\omega\in\Omega}|B_j(t,s)|<\infty$. For each
$i\geq1$, the process $C^{(i)}: \Delta^c\times \Omega\rightarrow
\mathbb{R}$ such that for each $t\in[0,T]$, $C^{(i)}(t, s)$ is
$F$-adapted, and $\sup_{(t,s)\in
\Delta^c}\mbox{esssup}_{\omega\in\Omega}|C^{(i)}(t,s)|<\infty$.

\bigskip

We now state the duality principle.
\begin{theorem}Suppose (H2) hold.
Let $X(t)$ be the solution of the following $\mathbb{R}$-valued
forward stochastic Volterra integral equation:
\begin{eqnarray}
X(t)&=&\psi(t)+\int_0^t
X(s)B_0(s,t)ds+\int_0^tX(s)B(s,t)dW_{s}\nonumber\\&&
+\sum_{i=1}^{\infty}\int_0^tX(s)C^{(i)}(s,t)dH^{(i)}_{s}.
\end{eqnarray}
Then, we have the following duality principle:
\begin{eqnarray}
E\int_0^T Y(t)\psi(t)dt=E\int_0^TX(t)\varphi(t)dt.
\end{eqnarray}
\end{theorem}
\begin{proof} From the definition of M-solution,
we have
\begin{eqnarray*}
&&E\int_0^T Y(t)\psi(t)
dt\\&=&E\int_0^TY(t)\left(X(t)-\int_0^tX(s)B_0(s,t)ds-\int_0^tX(s)B(s,t)dW_{s}\right.\\&&
\left.\qquad\qquad\qquad
-\sum_{i=1}^{\infty}\int_0^tX(s)C^{(i)}(s,t)dH^{(i)}_{s}\right) dt
\\&=&E\int_0^T Y(t)X(t) dt-E\int_0^T\int_s^TB_0(s,t)X(s)Y(t) dtds
\\&&
-E\int_0^T\left(\int_0^tX(s)B(s,t)dW_{s}+\sum_{i=1}^{\infty}\int_0^tX(s)C^{(i)}(s,t)dH^{(i)}_{s}\right)\times
\\&&\qquad\qquad \left(EY(t)+\int_0^t Z(t,s)dW_{s}+ \sum_{i=1}^{\infty}\int_0^tU^{(i)}(t,s)dH^{(i)}_{s}\right)dt
\\&=&
E\int_0^T Y(t)X(t) dt - E\int_0^T\int_t^T B_0(t,s)X(t)Y(s) dsdt
\\&&
-E\int_0^T \int_0^t  X(s)B(s,t)Z(t,s) dsdt
-\sum_{i=1}^{\infty}E\int_0^T \int_0^tX(s) C^{(i)}(s,t)U^{(i)}(t,s)
dtds\\&=& E\int_0^TY(t)X(t)dt -E\int_0^T\int_t^T B_0(t,s)X(t)Y(s)
dsdt
\\&&
-E\int_0^T \int_s^T X(s)B(s,t)Z(t,s)dtds
-\sum_{i=1}^{\infty}E\int_0^T \int_s^TX(s) C^{(i)}(s,t)U^{(i)}(t,s)
dtds
\\
&=& E\int_0^TY(t)X(t)dt -E\int_0^T\int_t^T B_0(t,s)X(t)Y(s) dsdt
\\&&
-E\int_0^T \int_t^T X(t)B(t,s)Z(s,t)dsdt
-\sum_{i=1}^{\infty}E\int_0^T \int_t^TX(t) C^{(i)}(t,s)U^{(i)}(s,t)
dsdt
\\&=& E\int_0^T \left[X(t) \left( Y(t)-\int_t^TB_0(t,s)Y(s)ds-\int_t^T  B(t,s) Z(s,t) ds\right.\right.\\&&
\left.\left.\qquad\qquad\qquad -\sum_{i=1}^\infty\int_t^T
C^{(i)}(t,s)U^{(i)}(s,t) ds\right)\right]dt
\\&=& E\int_0^T  X(t)\psi(t) dt.
\end{eqnarray*}
The proof is complete.
\end{proof}

With the help of the duality principle given above, we're able to
establish a comparison theorem for M-solutions of certain  BSVIELs.
Before we state the main result, we show the following Lemma.
\begin{lemma}\label{lemma:2}
Consider the following FSVIEL:
\begin{eqnarray}\label{bsde:20}
X(t)&=&g(t)+\int_0^t a(s, t)X(s)ds-\int_0^t b(s,
t)X(s)dW_{s}\nonumber\\&&-\sum_{i=1}^\infty\int_0^t c^{(i)}(s,
t)X(s)dH_s^{i}, \quad t\in[0,T],
\end{eqnarray}
where $a: [0,T]^2\times \Omega\rightarrow R$, $b: [0,T]^2\times
\Omega\rightarrow  \mathbb{R}^d$ and $c: [0,T]^2\times
\Omega\rightarrow \ell^2$ are three $\mathcal
{B}([0,T]^2)\otimes\mathcal {F}_T$-measurable and  uniformly bounded
processes, and for almost all $t\in[0,T]$, $a(t,\cdot), b(t,\cdot)$
and $c(t,\cdot)$ are $\mathbb{F}$-adapted. Moreover, for all
$(t,s,\omega)\in [0,T]^2\times\Omega$, $\sum_{i=1}^{\infty}
c^{(i)}(s) \Delta H^{(i)}_{s}>-1$.

 Then
for any $g(\cdot)\in \mathbb{L}^2_{\mathcal {F}_T}(0,T; R)$ with
$g(t)\geq 0$, we have
\begin{eqnarray}\label{bsde:25}
X(t)\geq 0, \quad t\in[0,T], a.s.
\end{eqnarray}
\end{lemma}
\begin{proof} The proof follows the ideas  in \cite{Yong06}. We first consider a special case of
FSVIEL (\ref{bsde:20}). More precisely, let $0=\tau_0<\tau_1<\cdots$
being a sequence of $\mathbb{F}$-stopping
 times, and
 \begin{eqnarray*}
a(s, t)&=&\sum_{k\geq 0}a_k(s){\bf 1}_{[\tau_k,\tau_{k+1}]}(t),\;
b(s, t)=\sum_{k\geq
0}b_k(s){\bf 1}_{[\tau_k,\tau_{k+1}]}(t),\\
c(s, t)&=&\sum_{k\geq 0}c_k(s){\bf
1}_{[\tau_k,\tau_{k+1}]}(t),\quad\; g(t)=\sum_{k\geq 0}g_k {\bf
1}_{[\tau_k,\tau_{k+1}]}(t),
\end{eqnarray*}
where  for all $k\geq 0$, $a_k(\cdot), b_k(\cdot)$ and $c_k(\cdot)$
being some $\mathbb{F}$-adapted and bounded processes such that
$\sum_{i=1}^{\infty} c_k^{(i)}(s) \Delta H^{(i)}_{s}>-1$, and each
$g_k\geq0$ is $\mathcal {F}_{\tau_k}$-measurable. As a result, on
$[0,\tau_1]$, the equation (\ref{bsde:20}) is equivalent to
\begin{eqnarray}\label{bsde:24}
X(t)&=&g_0+\int_0^t a_0(s)X(s)ds-\int_0^t
b_0(s)X(s)dW_{s}\nonumber\\&&-\sum_{i=1}^{\infty}\int_0^t
c_0^{(i)}(s)X(s)dH^{(i)}_{s},
\end{eqnarray}
From Protter \cite{Prot90}, the unique solution to equation
(\ref{bsde:24}) takes the form
\begin{eqnarray*}
X(t)&=&g_0\exp\left(\int_0^t((a_0(r)-b_0(r)^2/2)dr+dM_r)\right)\\&&\prod_{t<r\leq
 s}(1+\Delta M_r)\exp(-\Delta M_r)\geq 0,
\end{eqnarray*}
where
\begin{eqnarray*}
M_r=\int_0^rb_0(s)dW_{s}+\sum_{i=1}^{\infty}\int_0^rc_0^{(i)}(s)dH^{(i)}_{s}.
\end{eqnarray*}
By induction, we can prove that (\ref{bsde:25}) holds on $[\tau_i,
\tau_{i+1}]$. The general case can be proved by approximation.
\end{proof}

Next, we consider the following BSVIEL:
\begin{eqnarray}\label{bsde:19}
Y(t)&=&-\varphi(t)+\int_t^T f(t,s, Y(s-), Z(s,t),
U(s,t))ds\nonumber\\&&-\int_t^T
Z(t,s)dW_{s}-\sum_{i=1}^\infty\int_t^T U^{(i)}(t,s)dH_s^{i}, \,
t\in[0,T],
\end{eqnarray}
where the function $f: [0,T]^2\times \mathbb{R}\times
\mathbb{R}^d\times \ell^2 \rightarrow \mathbb{R}$ satisfies
assumption (\textbf{H1}) in a simplified way.

 As we know, the comparison theorem is not always hold for BSDEs with jump. One can
see Barles et al. \cite{BBP97} for a counterexample. In our frame,
we need the following extra assumption on the coefficient $f$:
\\(\textbf{H3}) the
function $f(t,s,y,z,u)$ is nondecreasing in $u$.
\begin{theorem}\label{theorem:4}
 Let $f, \overline{f}: [0,T]^2\times \mathbb{R}\times  \mathbb{R}^d\times
\ell^2\rightarrow \mathbb{R}$ satisfying (H1) and (H3) and let
$\varphi(\cdot), \overline{\varphi}(\cdot)\in \mathbb{L}^2_{\mathcal
{F}_T}(0, T; \mathbb{R})$ such that
\begin{eqnarray*}\label{bsde:21}
f(t,s,y,z,u)\leq \overline{f}(t,s,y,z,u), \forall
(t,s,y,z,u)\in[0,T]^2\times \mathbb{R}\times  \mathbb{R}^d\times
\ell^2, a.s.
\end{eqnarray*}
and
\begin{eqnarray}\label{bsde:22}
\varphi(t)\geq \overline{\varphi}(t), \forall t\in[0,T], a.s.
\end{eqnarray}
Let $(Y(\cdot),Z(\cdot,\cdot),U(\cdot,\cdot))$ (resp.
$(\overline{Y}(\cdot),\overline{Z}(\cdot,\cdot),\overline{U}(\cdot,\cdot))$)
be the adapted M-solution to BSVIEL (\ref{bsde:19}) corresponding to
$(f,\varphi)$ (resp. $(\overline{f},\overline{\varphi})$), then
\begin{eqnarray}\label{bsde:23}
Y(t)\leq\overline{Y}(t),\quad \forall t\in[0,T], a.s.
\end{eqnarray}
\end{theorem}
\begin{proof}
For $\forall t\in[0,T]$, we have
\begin{eqnarray}
&&Y(t)-\overline{Y}(t)\nonumber\\&=&\varphi(t)-\overline{\varphi}(t)+\int_t^T
f(t,s, Y(s-), Z(s,t), U(s,t))-\overline{f}(t,s,\overline{Y}(s),
\overline{Z}(s,t), \overline{U}(s,t))ds \nonumber\\&& -\int_t^T
[Z(t,s)-\overline{Z}(t,s)]dW_{s}-\sum_{i=1}^\infty\int_t^T
[U^{(i)}(t,s)-\overline{U}^{(i)}(t,s)]dH_s^{i}
 \nonumber\\
&=&\widehat{\varphi}(t)+\int_t^T \{B_0(t,s)[Y(s)-\overline{Y}(s)]
+B(t,s)[Z(t,s)-\overline{Z}(t,s)] \nonumber\\&&
 +\sum_{i=1}^\infty
C^{(i)}(t,s)[U^{(i)}(t,s)-\overline{U}^{(i)}(t,s)]\} ds -\int_t^T
[Z(t,s)-\overline{Z}(t,s)]dW_{s} \nonumber\\&&
-\sum_{i=1}^\infty\int_t^T
[U^{(i)}(t,s)-\overline{U}^{(i)}(t,s)]dH_s^{i},
\end{eqnarray}
where
\begin{eqnarray*}
\widehat{\varphi}(t)&=&\varphi(t)-\overline{\varphi}(t)+\int_t^T
f(t,s, \overline{Y}(s), \overline{Z}(s,t),
\overline{U}(s,t))\\&&\qquad\qquad\qquad\qquad\qquad-\overline{f}(t,s,\overline{Y}(s),
\overline{Z}(s,t), \overline{U}(s,t))ds\leq 0, \\
B_0(t,s)&=&[f(t,s, Y(s-), \overline{Z}(s,t),
\overline{U}(s,t))-f(t,s,\overline{Y}(s), \overline{Z}(s,t),
\overline{U}(s,t))]
\\&&
[Y(s)-\overline{Y}(s)]^{-1}{\bf 1}_{\{Y(s)\neq\overline{Y}(s)\}},
\end{eqnarray*}
and $B(t,s)=(B_1(t,s),\cdots,B_d(t,s))^T$,
$C(t,s)=(C^{(1)}(t,s),\cdots,C^{(i)}(t,s)\cdots)$. Here, for
$j=1,\cdots, d$,
\begin{eqnarray*}
B_j(t,s)&=&[f(t,s, \overline{Y}(s), \widehat{Z}_{j-1}(s,t),
\overline{U}(s,t))-f(t,s,\overline{Y}(s), \widehat{Z}_{j}(s,t),
\overline{U}(s,t))]\\&&[Z_{j}(s,t)-\overline{Z}_{j}(s,t)]^{-1}{\bf
1}_{\{Z_{j}(s,t)\neq\overline{Z}_{j}(s,t)\}},
\\ \widehat{Z}_{j}(s,t)&=&(\overline{Z}_1(s,t), \cdots, \overline{Z}_j(s,t), Z_{j+1}(s,t),
Z_d(s,t))
\end{eqnarray*}
and for $i\geq 1$,
\begin{eqnarray*}
C^{(i)}(t,s)&=&[f(t,s, \overline{Y}(s), \overline{Z}(s,t),
\widehat{U}^{(i-1)}(s,t))-f(t,s,\overline{Y}(s), \overline{Z}(s,t),
\widehat{U}^{(i)}(t,s))]\\&&
[U^{(i)}(s,t)-\overline{U}^{(i)}(s,t)]^{-1}{\bf
1}_{\{U^{(i)}(s,t)\neq\overline{U}^{(i)}(s,t)\}},
\\ \widehat{U}^{(i)}(t,s)&=&(\overline{U}^{(1)}(s,t), \cdots, \overline{U}^{(i)}(s,t), U^{(i+1)}(s,t),
\cdots).
\end{eqnarray*}
 From Lemma \ref{lemma:2}, we can prove the result.
\end{proof}

\section{Applications in Finance}

In this Section,  we  define a class of continuous-time dynamic risk
measures by means of BSVIELs.

The following definitions are borrowed from \cite{Yong07}.

\begin{definition}  A map $\rho: \mathbb{L}^2_{\mathcal
{F}_T}(0,T; \mathbb{R})\rightarrow \mathbb{L}^2_{\mathbb{F}}(0,T;
\mathbb{R})$ is called \emph{a dynamic risk measure} if the
following hold :

 (i)  For any $\varphi(\cdot),
 \overline{\varphi}(\cdot)\in \mathbb{L}^2_{\mathcal {F}_T}(0,T; \mathbb{R})$, if
$\varphi(s)=\overline{\varphi}(s), a.s. \omega\in\Omega, s\in[t,T]$
for some $t\in[0,T)$, then $ \rho(t; \varphi(\cdot))=\rho(t;
\overline{\varphi}(\cdot)), a.s. \omega\in\Omega. $

 (ii)
  For any $\varphi(\cdot),
 \overline{\varphi}(\cdot)\in \mathbb{L}^2_{\mathcal {F}_T}(0,T; \mathbb{R})$, if
$ \varphi(s)\geq \overline{\varphi}(s), a.s. \omega\in\Omega,
s\in[t,T] $ for some $t\in[0,T)$, then $ \rho(s; \varphi(\cdot))\leq
 \rho(s; \overline{\varphi}(\cdot)), a.s. \omega\in\Omega, s\in[t,T].
$
\end{definition}

\begin{definition} A dynamic risk measure $\rho: \mathbb{L}^2_{\mathcal
{F}_T}(0,T; \mathbb{R})\rightarrow \mathbb{L}^2_{\mathbb{F}}(0,T;
\mathbb{R})$ is called a \emph{coherent risk measure} if the
following hold :

(i)  There exists a deterministic integral function $r(\cdot)$ such
that for any  $\varphi(\cdot)\in \mathbb{L}^2_{\mathcal {F}_T}(0,T;
\mathbb{R})$,
\begin{eqnarray*}
\rho(t; \varphi(\cdot)+c)=\rho(t; \varphi(\cdot))-c
e^{-\int_t^Tr(s)ds}, a.s., t\in[0,T].
\end{eqnarray*}

(ii)  For any $\varphi(\cdot)\in \mathbb{L}^2_{\mathcal {F}_T}(0,T;
\mathbb{R})$ and $\lambda>0$,
\begin{eqnarray*}
\rho(t; \lambda\varphi(\cdot))=\lambda\varphi(t; \cdot), a.s.,
t\in[0,T].
\end{eqnarray*}

(iii)  For any $\varphi(\cdot),
 \overline{\varphi}(\cdot)\in \mathbb{L}^2_{\mathcal {F}_T}(0,T; \mathbb{R})$,
\begin{eqnarray*}
\rho(t; \varphi(\cdot) + \overline{\varphi}(\cdot))\leq\rho(t;
\varphi(\cdot))+\rho(t; \overline{\varphi}(\cdot)), a.s., t\in[0,T].
\end{eqnarray*}
\end{definition}
In what follows, we denote by
\begin{eqnarray}\label{bsde:26}
\rho(t; \varphi(\cdot))=Y(t),
\end{eqnarray}
where $(Y(\cdot),Z(\cdot,\cdot),
 U(\cdot,\cdot)$ is the unique M-solution of the following BSVIEL:
\begin{eqnarray}\label{bsde:29}
Y(t)&=&-\varphi(t)+\int_t^T f(t,s, Y(s-), Z(s,t),
U(s,t))ds\nonumber\\&&-\int_t^T
Z(t,s)dW_{s}-\sum_{i=1}^\infty\int_t^T U^{(i)}(t,s)dH_s^{i}, \,
t\in[0,T].
\end{eqnarray}

\begin{lemma}\label{lemma:4}
  Let $f:[0,T]^2\times \mathbb{R}\times  \mathbb{R}^d\times
\ell^2\rightarrow \mathbb{R}$ satisfy (H1) and (H3), suppose $f$ is
sub-additive, i.e.,
\begin{eqnarray*}
f(t,s,y_1+y_2, z_1+z_2, u_1+u_2)\leq  f(t,s,y_1, z_1,
u_1)+f(t,s,y_2, z_2, u_2),\\ \forall(t,s)\in [0,T]^2, y_1,y_2\in
\mathbb{R}, z_1,z_2\in\mathbb{R}^d, u_1, u_2\in \ell^2, a.e.
\end{eqnarray*}
then $\varphi(\cdot)\rightarrow\rho(t; \varphi(\cdot))$ is
sub-additive, i.e.,
\begin{eqnarray*}
\rho(t ;\varphi_1(\cdot)+\varphi_2(\cdot))\leq \rho(t
;\varphi_1(\cdot))+\rho(t ; \varphi_2(\cdot)), a.e.
\end{eqnarray*}
\end{lemma}
\begin{proof}
We can get the conclusion by Theorem \ref{theorem:4}.
\end{proof}

\begin{lemma}\label{lemma:5}
 (i) If the generator $f$ is of form
\begin{eqnarray*}
f(t,s,y,z,u)=r(s)y+\widetilde{f}(t,s,z, u),
\end{eqnarray*}
with $r(\cdot)$ being a deterministic integral function, then
$\varphi(\cdot)\rightarrow\rho(t; \varphi(\cdot))$ is transition
invariant, i.e.,
\begin{eqnarray*}
\rho(t; \varphi(\cdot)+c)=\rho(t; \varphi(\cdot))-c
e^{-\int_t^Tr(s)ds}, a.s., t\in[0,T], a.s., \forall c\in \mathbb{R}.
\end{eqnarray*}
In particular, if $r(\cdot)=0$, then
\begin{eqnarray*}
\rho(t; \varphi(\cdot)+c)=\rho(t; \varphi(\cdot))-c, a.s.,
t\in[0,T], a.s., \forall c\in \mathbb{R}.
\end{eqnarray*}

 (ii) If $f: [0,T]^2\times \mathbb{R}\times  \mathbb{R}^d\times
\ell^2\rightarrow \mathbb{R}$ is positively homogeneous, i.e., $
f(t,s,\lambda y,\lambda z,\lambda u)=\lambda f(t,s,y, z, u),
\forall(t,s)\in [0,T]^2, \lambda\in\mathbb{R}^+, y\in \mathbb{R},
z\in\mathbb{R}^d, u\in \ell^2, a.e. $ So is
$\varphi(\cdot)\rightarrow\rho(t; \varphi(\cdot))$.

\end{lemma}
\begin{proof}
The proof is obvious.
\end{proof}

By Lemmas \ref{lemma:4} and \ref{lemma:5}, we are able to construct
a class of dynamic coherent risk measures by means of solution of
certain BSVIELs. The proof of the following theorem is obvious, so
we omit it.

\begin{theorem}
Suppose $f$ satisfy (H1) and (H3). Moreover,
\begin{eqnarray*}
f(t,s,y,z,u)=r(s)y+\widetilde{f}(t,s,z, u),
\end{eqnarray*}
with $r(\cdot)$ being a bounded and  deterministic integral
function, then $\rho(\cdot)$ defined by (\ref{bsde:26}) is a dynamic
coherent risk measure if  $\widetilde{f}(t,s, z, u)$ is positively
homogeneous and sub-additive.
\end{theorem}

\end{document}